%% file: ms.tex
\pdfoutput=1
\documentclass[sigconf, acmthm=true]{acmart}
\usepackage[utf8]{inputenc} 

\renewcommand\footnotetextcopyrightpermission[1]{} 
\input{article-preamble}
\newcommand{\C}{\mathbb{C}}

\AtBeginDocument{%
	\providecommand\BibTeX{{%
			\normalfont B\kern-0.5em{\scshape i\kern-0.25em b}\kern-0.8em\TeX}}}

\begin{document}
	\title{Quantum Algorithms for Portfolio Optimization}
	
	\author{Iordanis Kerenidis}
	\email{jkeren@irif.fr}
	\affiliation{
		CNRS, IRIF, Universit\'e Paris Diderot, Paris, France
	}
	
	\author{Anupam Prakash}
	\email{anupam@irif.fr}
	\affiliation{
		CNRS, IRIF, Universit\'e Paris Diderot, Paris, France
	}
	
	\author{Dániel Szilágyi}
	\email{dszilagyi@irif.fr}
	\affiliation{
		CNRS, IRIF, Universit\'e Paris Diderot, Paris, France
	}
	
	\input{article-abstract}


	\maketitle
	
	\input{article-main}

	\bibliographystyle{ACM-Reference-Format}
	\bibliography{bibliography}

\end{document}

%% file: article-preamble.tex
\usepackage{amsmath, amsthm, amssymb, amsfonts, algorithmic, algorithm, hyperref, bm, subcaption}

\newcommand{\R}{\mathbb{R}}

\def\ket#1{\mathinner{|{#1}\rangle}}

\renewcommand{\part}[2]{\frac{\partial #1}{\partial #2}}

\newcommand{\all}[2]{\begin{align}\label{#2} #1\end{align}}

\newcommand{\poly}[1]{O(\text{poly} (#1))}

\newcommand{\xnext}{\vec{x}_\text{next}}
\newcommand{\snext}{\vec{s}_\text{next}}

\newcommand{\norm}[1]{\left\lVert#1\right\rVert}

\newcommand{\dx}{\Delta\vec{x}}
\newcommand{\dy}{\Delta\vec{y}}
\newcommand{\ds}{\Delta\vec{s}}

\newcommand{\dxm}{\measured{\dx}}

\newcommand{\dsm}{\measured{\ds}}

\newcommand{\Arw}{\operatorname{Arw}}

\newcommand{\interior}{\operatorname{int}}

\newcommand{\lorentz}{\mathcal{L}}

\renewcommand{\vec}[1]{\bm{#1}}
\newcommand{\vecrest}[1]{\widetilde{\vec{#1}}}

\newcommand{\measured}[1]{\overline{#1}}

\newcommand{\polylog}{\operatorname{polylog}}

\newcommand{\thmref}[1]{\hyperref[#1]{{Theorem~\ref*{#1}}}}
\newcommand{\lemref}[1]{\hyperref[#1]{{Lemma~\ref*{#1}}}}
\newcommand{\remref}[1]{\hyperref[#1]{{Remark~\ref*{#1}}}}
\newcommand{\corref}[1]{\hyperref[#1]{{Corollary~\ref*{#1}}}}
\newcommand{\eqnref}[1]{\hyperref[#1]{{Equation~(\ref*{#1})}}}
\newcommand{\claimref}[1]{\hyperref[#1]{{Claim~\ref*{#1}}}}
\newcommand{\remarkref}[1]{\hyperref[#1]{{Remark~\ref*{#1}}}}
\newcommand{\propref}[1]{\hyperref[#1]{{Proposition~\ref*{#1}}}}
\newcommand{\factref}[1]{\hyperref[#1]{{Fact~\ref*{#1}}}}
\newcommand{\defref}[1]{\hyperref[#1]{{Definition~\ref*{#1}}}}
\newcommand{\exampleref}[1]{\hyperref[#1]{{Example~\ref*{#1}}}}
\newcommand{\hypref}[1]{\hyperref[#1]{{Hypothesis~\ref*{#1}}}}
\newcommand{\secref}[1]{\hyperref[#1]{{Section~\ref*{#1}}}}
\newcommand{\chapref}[1]{\hyperref[#1]{{Chapter~\ref*{#1}}}}
\newcommand{\apref}[1]{\hyperref[#1]{{Appendix~\ref*{#1}}}}

%% file: article-abstract.tex
\begin{abstract}
We develop the first quantum algorithm for the constrained portfolio optimization problem. The algorithm has running time\\ $\widetilde{O} \left( n\sqrt{r} \frac{\zeta \kappa}{\delta^2} \log \left(1/\epsilon\right) \right)$, where $r$ is the number of positivity and budget constraints, $n$ is the number of assets in the portfolio, $\epsilon$ the desired precision, and $\delta, \kappa, \zeta$ are problem-dependent parameters related to the well-conditioning of the intermediate solutions. If only a moderately accurate solution is required, our quantum algorithm can achieve a polynomial speedup over the best classical algorithms with complexity $\widetilde{O} \left( \sqrt{r}n^\omega\log(1/\epsilon) \right)$, where $\omega$ is the matrix multiplication exponent that has a theoretical value of around $2.373$, but is closer to $3$ in practice. 

We also provide some experiments to bound the problem-dependent factors arising in the running time of the quantum algorithm,  and these experiments suggest that for most instances the quantum algorithm can potentially achieve an $O(n)$ speedup over its classical counterpart.
\end{abstract}

%% file: article-main.tex
\section{Introduction}
Quantum computation offers significant (even exponential) computational speedups over classical computation for a wide variety of problems \cite{S97, G96, HHL09}. 
The expected availability of small scale quantum computers in the near future has spurred interest in the development of applications for quantum computers that attain provable speedups over classical algorithms. 
Many such applications have been proposed in the field of quantum machine learning, including applications to principal components analysis \cite{LMR13}, clustering \cite{LMR13, KLLP18}, classification \cite{L19, KL18}, least squares regression \cite{KP17, CGJ18} and recommendation systems \cite{KP16}. Most of the quantum machine machine learning applications with provable speedups rely on quantum algorithms for linear algebra, utilizing some variant of the HHL algorithm \cite{HHL09} to obtain significant quantum speedups. 

Mathematical finance is an application area where quantum computers could potentially offer groundbreaking speedups. 
This is a very recent research area for quantum algorithms and is important in terms of applications as even modest speedups for computational financial problems can have enormous real world impact. Of course, translating these theoretical advantages of quantum algorithms into real world applications necesitates both much more advanced hardware, which may take some more years to come, but also a close collaboration between the communities of quantum algorithms and of mathematical Finance in order to really understand where and how such quantum algorithms can become a new powerful tool to be used within the general framework of mathematical finance. 

It has been suggested that quantum techniques like Feynman integrals could be useful for option pricing \cite{ B07}. 
There has also been experimental work where the IBM quantum computers have been used to explore quadratic speedups for option pricing \cite{S19, M19} and work on quadratic speedups for option pricing using Monte Carlo methods \cite{RGB18}. While some of these results lack provable guarantees, they indicate the strong interest in both the quantum algorithms and mathematical Finance communities in developing applications of quantum computers to computational finance. 

Very recently, Lloyd and Rebentrost \cite{LR18} proposed a quantum algorithm for the unconstrained portfolio optimization problem. Their algorithm uses quantum linear system solvers to obtain speedups for portfolio optimization problems that can be reduced to unconstrained quadratic programs, which in turn are reducible to a single linear system. 
The main limitation of their algorithm is that it can not incorporate positivity or budget constraints, thus restricting its applicability to real world problems that can have complex budget constraints. The reason for this limitation is algorithmic, the constrained portfolio optimization problem is known to be equivalent to quadratic programming (QP), a class of optimization problems that is more general than Linear programming (LP). 
Mathematical optimization problems, for example linear, quadratic and semidefinite programs (SDP) are not known to be reducible to a single instance of linear systems. All known algorithms for these optimization problems need to perform computations over a number of steps depending on the number of constraints. 

In this work, we design and analyse a quantum algorithm for the general constrained portfolio optimization problem. 
Our algorithm is applicable to portfolio optimization with an arbitrary number of positivity as well as budget constraints. We obtain a polynomial speedup over the classical algorithms and we provide experimental results to demonstrate the potential of these advantages in practice.

The main idea of our work is based on quantum algorithms for convex optimization.
In recent years, there has been work on quantum SDP (semidefinite program) solvers by quantizing the classical multiplicative weights method \cite{AG18, B+17} and the interior point methods \cite{KP18}. SDPs are more general than QPs so one may expect to obtain quantum speedups by applying some of these methods for QPs. However, the running time of all these quantum algorithms depends on a number of problem specific parameters and they do not achieve a worst case speedup over classical algorithms. Further, classical QP solvers are much more efficient than SDP solvers and there are special purpose QP algorithms that have complexity close to classical LP solvers. 

Our algorithm for constrained portfolio optimization first reduces the problem to second order cone programs (SOCP), a class of optimization problems that generalizes both LPs and QPs and has complexity close to LPs in the classical setting. We then use a recently developed quantum interior point method for SOCPs \cite{R19}, that extends the results obtained for LPs and SDPs \cite{KP18}. 

\section{Portfolio Optimization} 

The theory of portfolio optimization in mathematical finance was developed by Markovitz \cite{M52}. The theory describes how wealth can be optimally invested in assets which differ 
in expected return and risk. 

The input for the portfolio optimization problem is data about the historical prices of certain financial assets, the goal is to assemble the optimal portfolio that maximizes the expected return for a given level of risk. 
Let us assume that there are $m$ assets and we have data for historical returns on investment for $T$ time epochs for each asset. Let $R(t) \in \R^{m}$ be the vector of returns for all assets for time epoch $t$. The expected reward and risk for the assets can be estimated from the data as follows,  
\all{ 
\mu &= \frac{1}{T} \sum_{t \in [T]} R(t) \notag \\
\Sigma &= \frac{1}{T-1}  \sum_{t \in [T]} ( R (t) - \mu )   ( R (t) - \mu )^{T}
}{eq:one}
A portfolio is specified by the total investment $x_{j}$ in asset $j$ for all assets $j \in [m]$. The expected reward and risk for the portfolio $x$ are respectively given by $\mu^{T} x$ and $x^{T} \Sigma x$. 

The portfolio optimization can include positivity and budget constraints. Positivity constraints $x_{j} >0$ corresponds to the situation where it is not possible to short sell asset $j$. Budget constraints limit the amount of money invested in a subset of the assets and can depend on the value of the assets and also on the size of the initial investment. Similarly, one can add constraints to ensure that the portfolio is diversified by adding constraints on the amount of investment in certain asset classes. The constrained portfolio optimization problem can therefore capture a variety of complex constraints that arise in real world portfolio selection problems. 

In the quantum setting, we do not estimate the covariance matrix as in equation \eqref{eq:one}, we instead store the square root of the covariance $\Sigma$ to which we have direct access from the data, i.e. the covariance matrix $\Sigma = MM^{T}$ where $M$ is the matrix with columns $\frac{1}{ \sqrt{T-1}} (R(t) - \mu)$. Given the data it is easy to construct quantum data structures for operating with $M$ \cite{KP16}, hence we formulate the portfolio optimization problem in terms of $M$. 

The constrained portfolio optimization problem can be written as an optimization problem in one of several equivalent ways \cite{CT06}. We use the following formulation here:
\begin{equation}
\begin{array}{ll}
\min & \vec{x}^T M^{T}M \vec{x} \\
\text{s.t.}& \vec{\mu}^T \vec{x} = R  \\
& A\vec{x} = \vec{b} \\
& x \geq 0.
\end{array} \label{prob:portfolio}
\end{equation}
Note that for the above formulation, $\vec{x}$ is the portfolio, $\vec{\mu}$ is the vector of mean asset prices, $M$ is a matrix such that the covariance matrix of the asset prices is defined as $\Sigma= M^{T} M$, and $A\vec{x}=\vec{b}$ and $\vec{x} \geq 0$ are the constraints. An inequality constraint $C\vec{x} \geq \vec{d}$ can be realized by introducing a slack variable $\vec{s} := C\vec{x} - \vec{d}$ and requiring $\vec{s} \geq 0$. 

For the case when there are no inequality constraints, the problem becomes a linear least-squares problem, for which there is a closed-form solution \cite{LR18}. However, there is no closed form solution for the general constrained portfolio optimization problem. 

\section{Reducing portfolio optimization to SOCP} 
The portfolio optimization problem is typically reduced to quadratic programs (QP). We instead reduce it SOCPs (Second Order Cone Programs) which are a more general family of optimization problems. There are two reasons for using this reduction to SOCP besides the greater generality. First, the classical interior point algorithms for the SOCP is particularly well suited for being quantized, that is the linear systems generated for these methods are easily solved using quantum computers. Second, the reduction to SOCPs allows us to use powerful mathematical tools from the theory of Euclidean Jordan algebras for the analysis of our algorithm. In this section we provide the reduction of the constrained portfolio optimization problem to the specific form of SOCPs that enables the efficient quantum algorithm. 

The SOCP in the standard form is an optimization problem over the products of Lorentz or second order cones. 
\begin{definition} 
The $n$-dimensional Lorentz cone, for $n\geq 0$ is defined as
	\begin{equation*}
		\lorentz^n := \{ \vec{x} = (x_0; \vecrest{x}) \in \R^{n+1} \;|\; \norm{\vecrest{x}} \leq x_0 \}.
	\end{equation*}
\end{definition} 

 In other words the elements of the $n$-dimensional Lorentz cone are $n+1$ dimensional vectors where the first coordinate is an upper bound on the $\ell_2$-norm of the remainnig $n$-dimensional vector.

\noindent Formally, the SOCP in the canonical form is the following optimization problem: 
	\begin{equation*}
		\begin{array}{ll}
		\min & \sum_{i=1}^r \vec{c}_i^T \vec{x}_i\\
		\text{s.t.}& \sum_{i=1}^r A^{(i)} \vec{x}_i = \vec{b} \\
		& \vec{x}_i \in \lorentz^{n_i}, \forall i \in [r].
		\end{array}
	\end{equation*}
There are two important parameters in the definition of the SOCP, the first is the rank $r$ that is equal to the number of Lorentz cones in the product that we are optimizing over, the second $n = \sum_{i \in [r]} n_{i}$ is the number of variables. The running time for SOCP algorithms will be given in terms of $n$ and $r$. 
	
Concatenating the vectors $(\vec{x}_1; \dots; \vec{x}_r) =: \vec{x}$, $(\vec{c}_1; \dots; \vec{c}_r) =: \vec{c}$, and matrices $[A^{(1)} \cdots A^{(r)}] =: A$ horizontally, we can write the SOCP more compactly as an optimization problem over the cone $\lorentz := \lorentz^{n_1} \times \cdots \times \lorentz^{n_r}$ in the following way: 
	\begin{equation}
	\begin{array}{ll}
	\min & \vec{c}^T \vec{x}\\
	\text{s.t.}& A \vec{x} = \vec{b} \\
	& \vec{x} \in \lorentz,
	\end{array} \label{prob:SOCP primal}
	\end{equation}
	
We refer to this as the {\em primal} form of the SOCP and we also define the {\em dual} form as
		\begin{equation}
		\begin{array}{ll}
		\max & \vec{b}^T \vec{y}\\
		\text{s.t.}& A^T \vec{y} + \vec{s} = \vec{c}\\
		& \vec{s} \in \lorentz.
		\end{array} \label{prob:SOCP dual}
		\end{equation}
The problem in equation \eqref{prob:SOCP primal} is the SOCP \emph{primal}, and the one in \eqref{prob:SOCP dual} is its dual. A solution $(\vec{x}, \vec{y}, \vec{s})$ satisfying the constraints of both \eqref{prob:SOCP primal} and \eqref{prob:SOCP dual} is \emph{feasible}, and if in addition it satisfies $\vec{x} \in \interior \lorentz$ and $\vec{s} \in \interior \lorentz$, it is \emph{strictly feasible}. We assume that there exists a strictly feasible solution, for some cases there are methods for reducing a feasible problem to a strictly feasible one \cite{boyd2004convex}.
	
The constrained portfolio optimization problem can be reduced to an SOCP by using the matrix $M$ we defined earlier. First off, we see that $\vec{x}^T \Sigma \vec{x} = \norm{M\vec{x}}^2$, however, minimizing the squared norm is equivalent to minimizing the norm itself which turns out to be more naturally expressible using Lorentz constraints. We introduce an additional vector $\vec{t}:= (t_{0}, \vecrest{t})$ such that $\vecrest{t} = M\vec{x}$. Using this variable the portfolio optimization problem in equation \eqref{eq:one} is easily seen to be equivalent to the following SOCP in the canonical form, 
\begin{equation}
\begin{array}{ll}
\min & (1; 0^{n+m})^T (\vec{t}; \vec{x}) \\
\text{s.t.}&  \begin{bmatrix}
	0^m & -I_m & M \\
	0 & (0^m)^T & \mu^T \\
	0 & (0^m)^T  & A 
\end{bmatrix} \begin{bmatrix}
\vec{t} \\
\vec{x}
\end{bmatrix} = \begin{bmatrix}
0^m \\
R \\
b
\end{bmatrix} \\
& \vec{t} \in \lorentz^m, x_i \in \lorentz^0 \;\forall \; i \in [n],
\end{array} \label{prob:experimental portfolio SOCP}
\end{equation}
The optimal solution lies on the boundary of the cones and thus will have $t_0 = \norm{\vecrest{t}} = \norm{P\vec{x}}$, the SOCP objective function therefore also optimizes the objective function for the portfolio optimization problem \eqref{eq:one}. The remaining SOCP 
constraints respectively enforce the constraints $\vecrest{t} = M\vec{x}$, $\mu^{T} \vec{x} =R$ and $A\vec{x}=b$. The positivity constraints $\vec{x} \geq 0$ are ensured by the second order constraints $x_i \in \lorentz^0$. 

\section{The short step Interior-Point Method for SOCP} 
The quantum portfolio optimization problem is obtained by quantizing the classical short-step interior point method. In this section we describe the 
classical algorithm. The following definition of an arrowhead matrix will be useful for compactly stating the classical short-step interior 
point method. 

\begin{definition} 
(Arrowhead matrix.) For every vector $\vec{x}$, we can define the matrix (or \emph{linear}) representation of $\vec{x}$, $\Arw(\vec{x})$ as 
\[
	\Arw(\vec{x}) := \begin{bmatrix}
	x_0 & \vecrest{x}^T \\
	\vecrest{x} & x_0 I_n
	\end{bmatrix}.
\] 
\end{definition} 
In the classical short-step IPM \cite{MT00}, we start with a strictly feasible solution solution $(\vec{x}, \vec{y}, \vec{s})$ and update it in each step by solving a Newton linear system of size $O(n)$ that depends on the input parameters as well as the current duality gap $\nu = \frac{1}{r} \vec{x^{T}}\vec{s}$ and the current solution. Let $\sigma = (1 - 0.1/\sqrt{r})$, then the Newton linear system that is solved at each step of the interior point method is given by
\begin{align}
\begin{bmatrix}
A & 0 & 0 \\
0 & A^T & I \\
\Arw(\vec{s}) & 0 & \Arw(\vec{x}) 
\end{bmatrix}
\begin{bmatrix}
\dx \\
\dy \\
\ds
\end{bmatrix} = 
\begin{bmatrix}
\vec{b} - A \vec{x} \\
\vec{c} - \vec{s} - A^T \vec{y} \\
\sigma \nu \vec{e} - \Arw(\vec{x}) \vec{s}
\end{bmatrix}.
\label{eq:Newton system}
\end{align}
The solutions to the Newton linear system are used to update our current solution $(\vec{x}, \vec{y}, \vec{s})$ to $\vec{x} \gets \vec{x} + \dx$, $\vec{y} \gets \vec{y} + \dy$, $\vec{s} \gets \vec{s} + \ds$, that is guaranteed to have a duality gap that is smaller by a multiplicative factor of $\sigma$. The classical interior point method for SOCP is given as Algorithm \ref{alg:ipm}.
	
	\begin{algorithm}
		\caption{The interior point method for SOCPs } \label{alg:ipm} 
		\begin{algorithmic}[1]
			\REQUIRE Matrix $A$ and vectors $\vec{b}, \vec{c}$ in memory, precision $\epsilon>0$.  \\
			\begin{enumerate} 
				\item Find feasible initial point $(\vec{x}, \vec{y}, \vec{s}, \nu):=(\vec{x}, \vec{y}, \vec{s}, \nu_0)$.
				\item Repeat the following steps for $O(\sqrt{r} \log(\nu_0 / \epsilon))$ iterations. 
				\begin{enumerate}
					\item Solve the Newton system \eqref{eq:Newton system} to get $\dx, \dy, \ds$.
					\item Update $\vec{x} \gets \vec{x} + \dx$, $\vec{y} \gets \vec{y} + \dy$, $\vec{s} \gets \vec{s} + \ds$ and $\nu = \frac1r\vec{x}^T \vec{s}$. 
				\end{enumerate} 
				\item Output $(\vec{x}, \vec{y}, \vec{s})$.
			\end{enumerate}
		\end{algorithmic}
	\end{algorithm}

It can be shown that this algorithm ensures that the solutions remain strictly feasible and halves the duality gap every $O(\sqrt{r})$ iterations, so indeed, after $O(\sqrt{r} \log(\nu_0 / \epsilon))$ it will converge to a (feasible) solution with duality gap at most $\epsilon$ \cite{MT00}.
	
The time complexity for this algorithm is $O(\sqrt{r} n^{\omega} \log(\nu_0 / \epsilon))$ where $\omega < 2.37$ is the exponent for matrix multiplication. In practice the running time for the constrained portfolio optimization problem is $O(n^{3.5} \log(\nu_0 / \epsilon))$ as $r =O(n)$ if positivity constraints are included and the algorithms for matrix inversion requires time $O(n^{3})$ in practice. 
	
In practice, large scale portfolio optimization problems are solved using commercial \cite{andersen2000mosek} and open-source solvers \cite{borchers1999csdp, DCB13, T03} using interior-point methods (IPM). The complexity of the short-step IPM is dominated by the cost of solving linear systems. Thus, if we could improve the time needed to solve a linear system, we would also improve the complexity of our IPM by the same factor.
	
Quantum computation offers exactly such a speedup for solving linear systems with sublinear algorithms for 'solving' linear systems. Starting with the work of \cite{HHL09}, it has become possible to solve a (well-conditioned) linear system in time that can be polylogarithmic in the matrix dimensions. This basic technique has been improved significantly, as we will discuss in the following section. 
	
Of course, when solving $A\vec{x} = \vec{b}$, $\vec{x} \in \R^n$, it is not even possible to write down all $n$ coordinates of the solution vector $\vec{x}$ in time $o(n)$. Instead, these algorithms encode vectors as quantum states, so that $\vec{z} \in \R^n$ with $\norm{\vec{z}}=1$ is encoded as the quantum state
\begin{equation}\label{eq:quantum vector notation}
\ket{\vec{z}} = \sum_{i=1}^{n} z_i \ket{i},
\end{equation}
where we write $\ket{i}$ for the joint $\left\lceil \log_2(n) \right\rceil$-qubit state corresponding to $\left\lceil \log_2(n) \right\rceil$-bit binary expansion of $i$. 
	
Then, the solution they output is a quantum state $\ket{\phi}$ close to $\ket{A^{-1}\vec{b}}$. In case a ``real'' (classical) solution is needed, then one can perform \emph{tomography} and \emph{norm estimation} on $\ket{\phi}$, and obtain a classical vector $\measured{\vec{x}}$ that is close to $\ket{\phi}$, so that finally we have a guarantee $\norm{\vec{x} - \measured{\vec{x}}} \leq \epsilon \norm{\vec{x}}$ for some $\epsilon > 0$.

\section{Our Results} 
Our main result is a quantum algorithm for the constrained portfolio optimization problem with the following convergence and running time guarantees. 

\begin{theorem}\label{thm:runtime}
	There is a quantum algorithm for the constrained portfolio optimization problem \ref{prob:portfolio} that outputs a solution $\vec{x}$ with the 
	following guarantees. \\
	(i) The objective value for $\vec{x}$ is within $\epsilon r$ of the optimum.  \\
	(ii) The solution satisfies the inequality constraints $\vec{x} \geq 0$, the equality constraints are satisfied approximately to precision $\delta$, that is $\norm{A\vec{x}_T - \vec{b}} \leq \delta\norm{A}$. \\
	(iii) The running time for the algorithm is
	\begin{equation*}
		T = \widetilde{O} \left( \sqrt{r} \log\left( n / \epsilon \right) \cdot \frac{n \kappa \zeta}{\delta^2}\log\left( \frac{\kappa \zeta}{\delta} \right)  \right).
	\end{equation*}
	where $\kappa = \max_{ i \in [T]} \kappa_{i}$ is the maximum condition number of the Newton matrix \eqref{eq:Newton system} over the $T= O(\sqrt{r} \log\left( n / \epsilon \right))$ iterations and $\zeta \leq \sqrt{n}$ is a parameter that appears in the quantum linear system solver. 
\end{theorem}

The quantum portfolio optimization algorithm is obtained using the reduction to the SOCP \eqref{prob:experimental portfolio SOCP} and then using a quantum short-step quantum interior point method for SOCPs described in \cite{R19} that has the following guarantees. 

\begin{theorem}\label{thm:runtimeS} \cite{R19} 
	There is a quantum algorithm that given a primal-dual pair of SOCPs \eqref{prob:SOCP dual} outputs solutions $(\vec{x},\vec{y},\vec{s})$ with the following guarantees.\\
	(i) The duality gap $\frac{1}{r} \vec{x}^{T} \vec{s} \leq \epsilon$.  \\
	(ii) The equality constraints are satisfied to precision $\delta$, that is\\$\norm{A\vec{x}_T - \vec{b}} \leq \delta\norm{A}$ and $\norm{A^T \vec{y}_T + \vec{s}_T - \vec{c}} \leq \delta \left( \norm{A} + 1 \right)$. \\
	(iii) The running time for the algorithm is
	\begin{equation*}
		T = \widetilde{O} \left( \sqrt{r} \log\left( n / \epsilon \right) \cdot \frac{n \kappa \zeta}{\delta^2}\log\left( \frac{\kappa \zeta}{\delta} \right)  \right).
	\end{equation*}
	where $\kappa = \max_{ i \in [T]} \kappa_{i}$ is the maximum condition number of the Newton matrix \eqref{eq:Newton system} over the $T= O(\sqrt{r} \log\left( n / \epsilon \right))$ iterations and $\zeta \leq \sqrt{n}$ is a parameter that appears in the quantum linear system solver. 
\end{theorem}

The quantum portfolio algorithm can achieves a significant polynomial speedup over the classical $O(n^{3.5} \log (n/\epsilon))$ time algorithm if the condition number $\kappa$ of the Newton matrices arising in the classical short-step interior point method is bounded. 
In Section \ref{sec:Numerical results} we present some numerical experiments that give some indications on how this parameter behaves for real-world problems. The experiments suggest that this parameter $\kappa$ in indeed bounded and that our algorithm achieves a speedup over the corresponding classical algorithm. 	

Let us make some observations comparing the main result with the classical short step interior point method. The quantum algorithm uses quantum linear algebra algorithms and tomography in order to implement a single step of the classical algorithm. As the solutions to the linear system are reconstructed using quantum tomography and are not exact, unlike the classical case the solutions may not be exactly feasible, the error parameter $\delta$ in part (ii) of Theorem \ref{thm:runtime} corresponds to the error induced by quantum tomography. 

The number of iterations for both the classical and quantum algorithms is $O \left( \sqrt{r} \log\left( n / \epsilon \right) \right)$. The analysis shows that for the quantum algorithm, the duality gap reduced by a factor $1 - \alpha/\sqrt{r}$ in each step where $\alpha <0.1$ is a constant. That is, the duality gap for the quantum algorithm decreases at a slower rate than that for the classical algorithm but the asymptotic decay rate is the same, implying that both algorithms converge in $O \left( \sqrt{r} \log\left( n / \epsilon \right) \right)$ iterations. Experiments suggest that the constant $\alpha$ is fairly close to the constant $\chi = 0.1$ for the classical algorithms, the number of iterations needed for convergence is therefore similar for both algorithms. 

This complexity of the quantum algorithm can be easily interpreted as product of the number of iterations and the cost of $n$-dimensional vector tomography with error $\delta$. Note that the same Newton linear system of dimension $O(n)$ is being solved by both the quantum and classical algorithms. The quantum portfolio optimization problem is particularly suited to the case where the number of stocks $n$ is large while number of budget constraints $r$ is small. This is because the factor $\kappa$ in the running time increases with the number of iterations, so quantum computers offer the maximum advantage over classical algorithms when the linear system size $n$ is large and the number of iterations for the IPM which scales as $\widetilde{O}(\sqrt{r})$ is small. 

\section{The quantum short step IPM} 
In this section we describe the quantum portfolio optimization algorithm. We begin by stating the state of the art quantum linear algebra results in Section \ref{qla}. The quantum short-step interior point method is then presented  in Section \ref{qipm}. 

\subsection{Quantum linear algebra} \label{qla} 
Given a matrix $A$ and a vector $\vec{b}$, the quantum linear system problem is to construct the quantum state $\ket{A^{-1}\vec{b}}$ (using the notation from \eqref{eq:quantum vector notation}). Once we have the state $\ket{x} = \ket{A^{-1} \vec{b}}$, can use either it in further computations or sample from the corresponding probability distribution. We can also perform \emph{tomography} on it to recover the underlying classical vector. 
If we perform tomography and norm estimation we obtain a classical vector $\measured{\vec{x}} \in \R^n$ that satisfies $\norm{\vec{x} - \measured{\vec{x}}} \leq \delta \norm{\vec{x}}$, where $A\vec{x} = \vec{b}$ and $A \in \R^{n \times n}$, $\vec{b} \in \R^n$. In this section we describe the state of the art quantum linear algebra algorithms for performing these operations. 

In this section, we also assume that $A$ is symmetric, since otherwise we can work with its symmetrized version $\operatorname{sym}(A) = \begin{bmatrix}
		0 & A \\
		A^T & 0 
\end{bmatrix}$. The quantum linear system solvers from \cite{CGJ18,GLSW18} require access to an efficient \emph{block encoding} for $A$, which is defined as follows. 
	\begin{definition} \label{use} 
		Let $A \in \R^{n\times n}$ be a matrix. Then, the $\ell$-qubit unitary matrix $U \in \C^{2^\ell \times 2^\ell}$ is a $(\zeta, \ell)$ block encoding of $A$ if $U = \begin{bmatrix}
		A / \zeta & \cdot \\
		\cdot & \cdot
		\end{bmatrix}$.
	\end{definition}
In order to perform quantum linear algebra operations for $A$ we need access to an efficient implementation of a unitary block encoding for $A$. That is we require that the unitary $U$ in definition \ref{use} be implemented efficiently, i.e. using an $\ell$-qubit quantum circuit of depth (poly)logarithmic in $n$. It turns out efficient construction and implementation of unitary block encodings for arbitrary matrices can be obtained using data structures for storing the matrices in QRAM (Quantum Random Access Memory) that were proposed in  \cite{KP16, KP17}. Moreover, these data structure can be constructed in linear time from a classical description of $A$ and support efficient updates. 
	\begin{theorem}[Block encodings using QRAM data structures \cite{KP16, KP17}] \label{qbe}
		There exist QRAM data structures for storing vectors $\vec{v}_i \in \R^n$, $i \in [m]$ and matrices $A \in \R^{n\times n}$ such that with access to these data structures one can do the following:
		\begin{enumerate}
			\item Given $i\in [m]$, prepare the state $\ket{\vec{v}_i}$ in time $\widetilde{O}(1)$. In other words, the unitary $\ket{i}\ket{0} \mapsto \ket{i}\ket{\vec{v}_i}$ can be implemented efficiently.
			\item A $(\zeta(A), 2 \log n)$ unitary block encoding for $A$ with $\zeta(A) = \min( \norm{A}_{F}/\norm{A}_{2}, s_{1}(A)/\norm{A}_{2})$, where $s_1(A) = \max_i \sum_j |A_{i, j}|$ can be implemented in time $\widetilde{O}(\log(n))$. Moreover, this block encoding can be constructed in a single pass over the matrix $A$, and it can be updated in $O(\log^{2} n)$ time per entry.
		\end{enumerate}
	\end{theorem}
	Note that $\ket{i}$ is the notation for the $\left\lceil \log(m) \right\rceil$ qubit state corresponding to the binary expansion of $i$. The QRAM can be thought of as the quantum analogue to RAM, i.e. an array $[b^{(1)}, \dots, b^{(m)}]$ of $w$-bit bitstrings, whose elements we can access given their address (position in the array). More precisely, QRAM is just an efficient implementation of the unitary transformation 
	\begin{equation*}
		\ket{i}\ket{0}^{\otimes w} \mapsto \ket{i} \left( \ket{b^{(i)}_1} \otimes \cdots \otimes \ket{b^{(i)}_w} \right), \text{ for } i \in [m].
	\end{equation*}
	Nevertheless, from now on, we will also refer to storing vectors and matrices in QRAM, meaning that we use the data structure from Theorem \ref{qbe}. Once we have these block encodings, we may use them to perform linear algebra:
	\begin{theorem} \label{qlsa} 
		(Quantum linear algebra with block encodings) \cite{CGJ18, GLSW18} Let $A \in \R^{n\times n}$ be a matrix with non-zero eigenvalues in the interval $[-1, -1/\kappa] \cup [1/\kappa, 1]$, and let $\epsilon > 0$. Given an implementation of an $(\zeta, O(\log n))$ block encoding for $A$ in time $T_{U}$ and a procedure for preparing state $\ket{b}$ in time $T_{b}$, 
		\begin{enumerate} 
			\item A state $\epsilon$-close to $\ket{A^{-1} b}$ can be generated in time 
			$O((T_{U} \kappa \zeta+ T_{b} \kappa) \polylog(\kappa \zeta /\epsilon))$. 
			\item A state $\epsilon$-close to $\ket{A b}$ can be generated in time $O((T_{U} \kappa \zeta+ T_{b} \kappa) \polylog(\kappa \zeta /\epsilon))$. 
			\item For  $\mathcal{A} \in \{ A, A^{-1} \}$, an estimate $\Lambda$ such that $\Lambda \in (1\pm \epsilon) \norm{ \mathcal{A} b}$ can be generated in time $O((T_{U}+T_{b}) \frac{ \kappa \zeta}{ \epsilon} \polylog(\kappa \zeta/\epsilon))$. 
		\end{enumerate}
	\end{theorem}  
	\noindent Finally, in order to recover classical information from the outputs of a linear system solver, we require an efficient procedure for quantum state tomography. The tomography procedure is linear in the dimension of the quantum state. 
	\begin{theorem}[Efficient vector state tomography, \cite{KP18}]\label{vector state tomography}
		Given a unitary mapping $U:\ket{0} \mapsto \ket{\vec{x}}$ in time $T_U$ and $\delta > 0$, there is an algorithm that produces an estimate $\measured{\vec{x}} \in \R^d$ with $\norm{\measured{\vec{x}}} = 1$ such that $\norm{\vec{x} - \measured{\vec{x}}} \leq  \delta$ with probability at least $(1 - 1/d^{0.83})$ in time $O\left( T_U \frac{d \log d}{\delta^2}\right)$.
	\end{theorem}
	Repeating this algorithm $\widetilde{O}(1)$ times allows us to increase the success probability to at least $1 - 1/\poly n$. Putting together Theorems \ref{qbe}, \ref{qlsa} and \ref{vector state tomography}, we obtain that there is a quantum algorithm that outputs a classical solution to the linear system $A\vec{x} = \vec{b}$ to accuracy $\delta$ in the $\ell_{2}$-norm in time $\widetilde{O} \left( n \cdot \frac{\kappa \zeta}{\delta^2} \right)$. For well-conditioned matrices, this presents a significant speedup over $O(n^\omega)$ needed for solving linear systems classically, especially for large $n$ and moderately small values of $\delta$. 

\subsection{The Quantum short-step IPM} \label{qipm} 
The quantum portfolio optimization is described as Algorithm \ref{alg:qipm}. It first computes from the dataset the vector $\mu$ and matrix $M$ as required for the portfolio optimization problem formulated in equation \eqref{prob:portfolio}.  Using $M$ and $\mu$, it then computes the constraint matrices for the SOCP \eqref{prob:experimental portfolio SOCP} equivalent the portfolio optimization. It then invokes the short-step quantum interior point method for SOCPs \cite{R19} and follows the steps there to obtain a solution to the SOCP. The solution to the portfolio optimization problem is then recovered from the SOCP solution. 
The construction of block encodings for the Newton matrices arising in the portfolio optimization problem is much easier than that for general SDPs. 
\begin{algorithm} 
		\caption{The quantum interior point method for Portfolio optimization. } \label{alg:qipm} 
		\begin{algorithmic}[1]
			\REQUIRE Dataset consisting of historical prices for $n$ financial assets.    \\
			\begin{enumerate} 
				\item Compute vector $\mu$ and matrix $M$ as in equation \eqref{eq:one} from the dataset.
				\item Store the constraint matrix for the SOCP formulation \eqref{prob:experimental portfolio SOCP} of the portfolio 
				optimization problem in the QRAM using the matrices from Step 1. 
				\item The remaining steps of this algorithm solve this SOCP using the short-step IPM from \cite{R19}. Repeat the following steps for $T$ iterations. 
				\begin{enumerate}
					\item Construct the block encoding for Newton linear system matrix \eqref{eq:Newton system} for the SOCP. 
					
					\hspace{-1.1cm} {\textbf{Estimate} $\left( \dx ; \dy ; \ds \right)$}\\
					
					\item {\textit{Estimate norm of} $\left( \dx ; \dy ; \ds \right)$.}  
					
					Solve the Newton linear system and perform norm estimation as in Theorem \ref{qlsa} to obtain $\measured{\norm{\left( \dx ; \dy ; \ds \right)}}$ such that 
					with probability $1-1/\poly n$, 
					\[ \measured{ \norm{\left( \dx ; \dy ; \ds \right)} }  \in (1 \pm \delta)  \norm{\left( \dx ; \dy ; \ds \right)}.     \]
					
					\item {\em Estimate $\left( \dx ; \dy ; \ds \right)$.}
					
					Let $U_N$ the procedure that solves the Newton linear system to produce states $\ket{\left( \dx ; \dy ; \ds \right)}$ to accuracy $\delta^2/n^3$.\\
					Perform tomography with $U_N$ (Theorem \ref{vector state tomography}) and use the norm estimate from (b) to obtain the classical estimate $\measured{\left( \dx ; \dy ; \ds \right)}$ such that with 
					probability $1-1/\poly n$, 
					\[\norm{ \measured{\left( \dx ; \dy ; \ds \right)} - \left( \dx ; \dy ; \ds \right)} \leq 2\delta \norm{\left( \dx ; \dy ; \ds \right)}.\]
					\hspace{-1.1cm} {\textbf{Update solution}}\\
					\item Update $\vec{x} \gets \vec{x} + \dxm$, $\vec{s} \gets \vec{s} + \dsm$ and store in QRAM. \\Update $\nu \gets \frac1r \vec{x}^T \vec{s}$. 
				\end{enumerate} 
				\item At the end of $T$ iterations we have the SOCP solution $(\vec{x}, \vec{y}, \vec{s})$. Output $x$ as the solution to the portfolio optimization problem. 
			\end{enumerate}
		\end{algorithmic}
	\end{algorithm}

The analysis of the quantum portfolio optimization algorithm follows from the analysis of a single iteration of the approximate short-step IPM for SOCPs \cite{R19}, from which the runtime and convergence results follow as consequences: 
	\begin{theorem} \label{thm:main} \cite{R19} 
		Let $\chi = \eta = 0.01$ and $\xi = 0.001$ be positive constants and let $(\vec{x}, \vec{y}, \vec{s})$ be solutions of \eqref{prob:SOCP primal} and \eqref{prob:SOCP dual} with $\nu = \frac1r \vec{x}^T \vec{s}$ and $d(\vec{x}, \vec{s}, \nu) \leq \eta\nu$. Then, the Newton system \eqref{eq:Newton system} has a unique solution $(\dx, \dy, \ds)$. If we assume that $\dxm, \dsm$ are approximate solutions of \eqref{eq:Newton system} that satisfy
		\begin{align*} 
		\norm{\dx - \dxm}_F &\leq \frac{\xi}{\norm{ (2\Arw^{2}(\vec{x}) - \Arw(\Arw(\vec{x}) \vec{x}))^{-1/2}   }} \text{ and } \\
		\norm{\ds - \dsm}_F &\leq \frac{\xi}{\norm{ (2\Arw^{2}(\vec{s}) - \Arw(\Arw(\vec{s}) \vec{s}))^{-1/2}   }} 
		\end{align*} 
		If we let $\xnext := \vec{x} + \dxm$ and $\snext := \vec{s} + \dsm$, the following holds:
		\begin{enumerate}
			\item The updated solution is strictly feasible, i.e. $\xnext \in \interior \lorentz$ and $\snext \in \interior \lorentz$.
			\item The updated solution satisfies $d(\xnext, \snext, \measured{\nu}) \leq \eta\measured{\nu}$ and \\$\frac1r \xnext^T \snext = \measured{\nu}$ for $\measured{\nu} = \measured{\sigma}\nu$, $\measured{\sigma} = 1 - \frac{\alpha}{\sqrt{r}}$ and a constant $0 < \alpha \leq \chi$.
		\end{enumerate}
	\end{theorem}

The proof of the Theorem \ref{thm:main} is similar to the analysis of interior point methods with approximation errors incorporated at each step. The analysis for SOCPs is carried out using the mathematical framework of Euclidean Jordan algebras that provides a unified analysis for LPs, SOCPs and SDPs generalizing the analysis known for the case of the exact interior point methods in the optimization literature \cite{R19}. 

In order to complete the analysis of the quantum portfolio optimization algorithm, in addition to the above Theorem, we need a bound on the infeasibility of the linear constraints $A\vec{x} = \vec{b}$ that are not exactly satisfied due to tomography errors. It turns out that this error is not accumulated, but is instead determined just by the final tomography precision. This is an improvement upon the quantum interior point method for SDPs proposed \cite{KP18} where the errors accumulate leading to worse parameters for the accuracy of the solutions. 
	\begin{theorem}
		Let \eqref{prob:SOCP primal} be a SOCP as in Theorem \ref{thm:runtime}. Then, after $T$ iterations, the (linear) infeasibility of the final iterate $\vec{x}_{T}$  is bounded as
		\begin{align*}
		\norm{A\vec{x}_T - \vec{b}} &\leq \delta\norm{A}. 
		\end{align*}
	\end{theorem}
	\begin{proof}
		Let $\vec{x}_T$ be the solution at the $T$-th iterate. Then, the following holds for $A\vec{x}_T - \vec{b}$:
		\begin{equation}\label{eq:eq11}
		A \vec{x}_T - \vec{b} = A\vec{x}_0 + A\sum_{t=1}^{T} \dxm_t - \vec{b} = A \sum_{t=1}^T \dxm_t.
		\end{equation}
		On the other hand, the Newton system at iteration $T$ has the constraint $A\dx_T = \vec{b} - A\vec{x}_{T-1}$, which we can further recursively transform as, 
		\begin{align*}
		A\dx_T &= \vec{b} - A\vec{x}_{T-1} = \vec{b} - A\left( \vec{x}_{T-2} + \dxm_{T-1} \right) \\
		&= \vec{b} - A\vec{x}_0 - \sum_{t=1}^{T-1} \dxm_t = - \sum_{t=1}^{T-1} \dxm_t.
		\end{align*}
		Substituting this into equation \eqref{eq:eq11}, we get
		\[
		A\vec{x}_T - \vec{b} = A \left( \dxm_T - \dx_T \right).
		\]
		
		Finally, we can bound the norm of this quantity, 
		\begin{align*}
		\norm{A\vec{x}_T - \vec{b}} &\leq \delta\norm{A}. 
		\end{align*}
	\end{proof}
	
\section{Experimental results} \label{sec:Numerical results} 
	In the experiments, we solve the following portfolio problem
\begin{equation}
\begin{array}{ll}
\min & \vec{x}^T \Sigma \vec{x} \\
\text{s.t.}& \vec{\mu}^T \vec{x} = R \\
& x \geq 0,
\end{array} \label{prob:experimental portfolio}
\end{equation}
i.e. just the problem \eqref{prob:portfolio} without the ``complicated'' linear constraints.
The dataset was obtained from cvxportfolio repository \cite{BB17}, this dataset contains historical data about the stocks of the S\&P-500 companies for each day over a period of 9 years (2007-2016). We sub-sampled the dataset to 50 companies and considered the stock performance for the first 100 days for our experiments. We computed the optimal portfolio for this dataset. 

In order to simulate the quantum algorithm, we added a Gaussian noise of magnitude corresponding to the tomography precision in the convergence Theorem \ref{thm:main}. The tomography precision in fact has a mathematical interpretation and is the Jordan algebra analog of the minimum eigenvalue for the solution matrices in the SDP case. We plotted these parameters for simulations of the classical algorithm where the noise is $0$ and for the quantum algorithm where the noise is chosen according to the tomography precision as described above. The plots are illustrated in the Figure \ref{fig:mu delta evolution}. 

\begin{figure}[h]
  \centering
  \includegraphics[width=\linewidth]{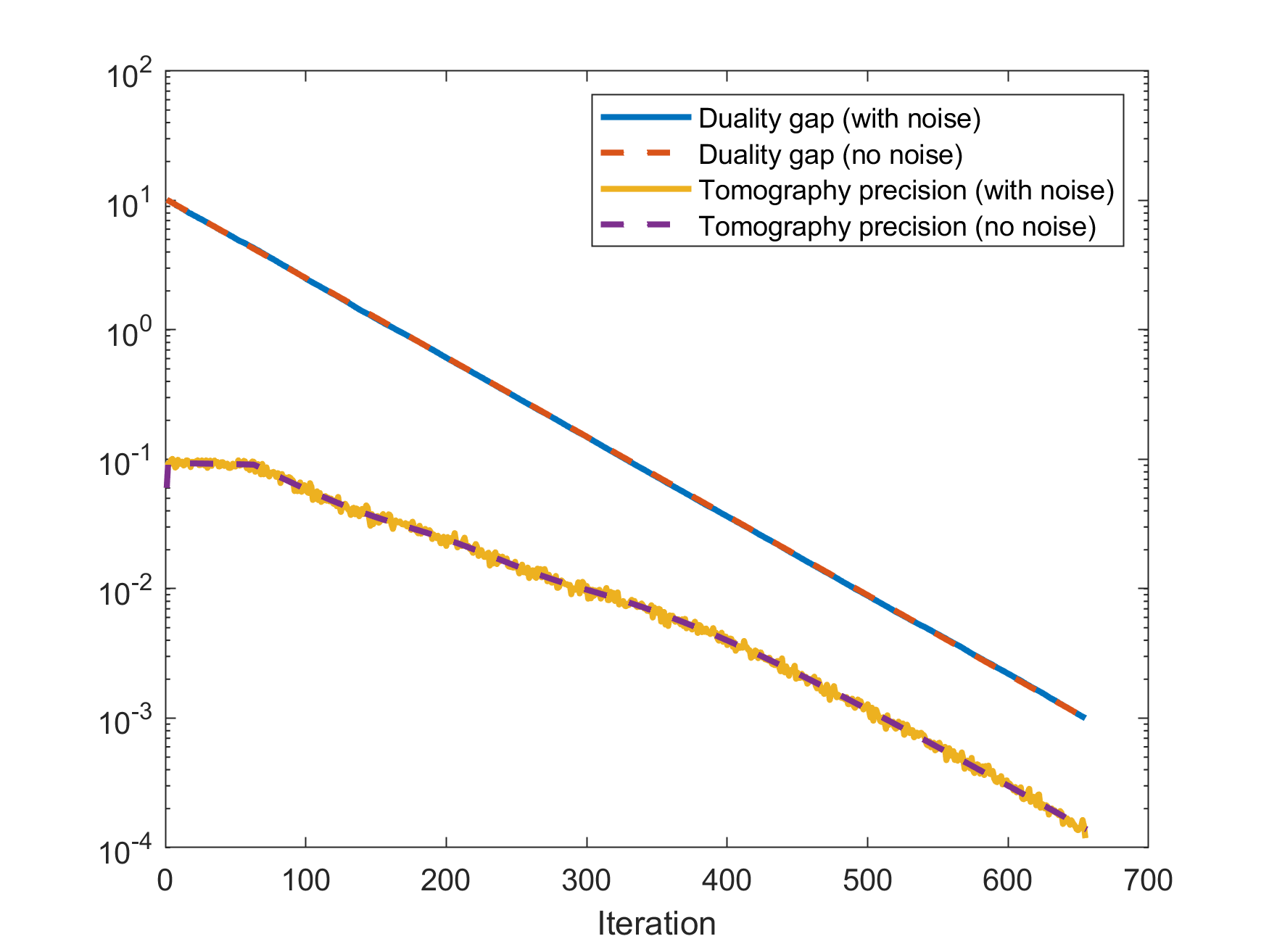}
  \caption{Duality gap and tomography precision vs iteration count for classical and quantum algorithms.}
  \label{fig:mu delta evolution}
\end{figure}

Figure \ref{fig:mu delta evolution} shows that the duality gap $\nu$ decreases multiplicatively by a factor of $\sigma$ as predicted by the analysis for both the classical and quantum algorithms. The rate of decrease for the quantum algorithm is close to that for the classical algorithm, that is the constant $\alpha$ that determines the rate of decrease $(1- \alpha/\sqrt{n})$ for the quantum algorithm is close to the constant $\chi = 0.1$ for the classical algorithm. After a certain point, the required tomography precision in the convergence Theorem \ref{thm:main} is close to the duality gap and shows a similar multiplicative decrease. 
Thus, the experiments suggest that the required tomography precision is indeed bounded in terms of duality gap and does not contribute a large factor to the running time. 

Figure \ref{fig:kappa} plots the condition number $\kappa$ of the Newton matrix against the iteration count. Even though theory \cite{dollar2005iterative} suggests  that the condition number $\kappa$ grows as the inverse of the duality gap $\nu$, it seems that in practice this upper bound is not attained. In fact, experiments suggest that the growth for the condition number over $T$ iterations is $O\left( 1 / \nu^{\alpha} \right)$ for $\alpha < 0.5$. 

\begin{figure}[h]
  \centering
  \includegraphics[width=\linewidth]{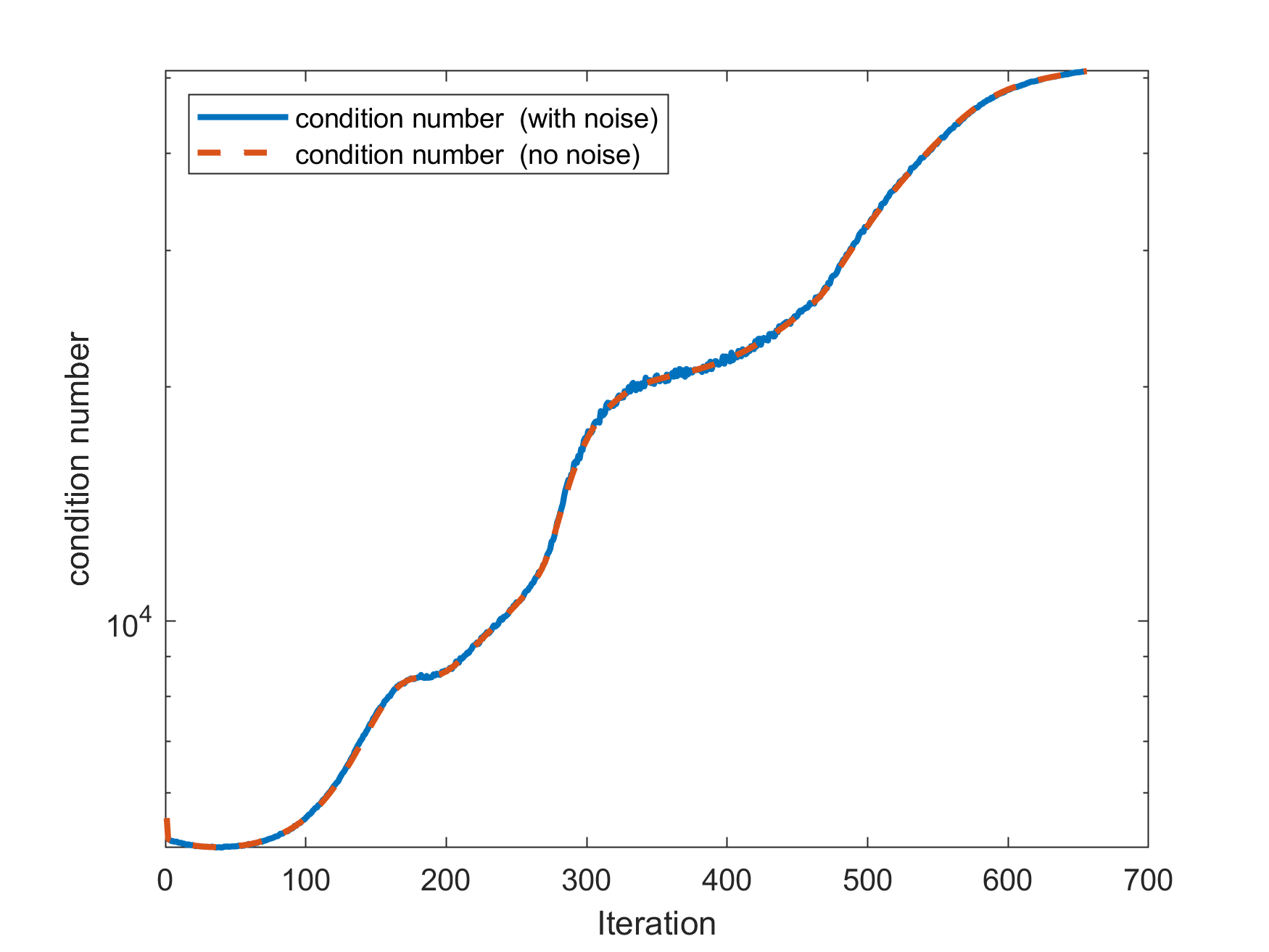}
  \caption{Condition number of Newton matrix vs iteration count for classical and quantum algorithms.}
  \label{fig:kappa}
\end{figure}

Now, having described the dependence of the precision $\delta$ and the condition number $\kappa$ on the duality gap $\nu$ (i.e. the precision parameter $\epsilon$), we investigate their dependence on the size parameter $n$ (the size of the Newton linear system). In order to do that, we sampled random subsets of the cvxportfolio dataset, solved the corresponding instance of problem \eqref{prob:experimental portfolio} using Algorithm \ref{alg:qipm} up to a fixed precision $\epsilon$, and recorded the final (worst) values of $\kappa, \delta$ and $\zeta$. More precisely, the instances of \eqref{prob:experimental portfolio} were constructed by choosing 100 random companies, choosing a random subinterval of $t$ days (such that $t$ is uniform in $[10, 500]$), and setting the precision parameter to $\epsilon = 0.1$. Fixing $\epsilon$ to such a value is justified because the stochasticity of financial markets outweighs by far the accumulated numerical errors. Immediately, we see that $\zeta < 5$ for all instances, and since $\zeta$ is always at least 1, we conclude that it is not very important for the final runtime of Algorithm \ref{alg:qipm}.

\begin{figure}[h]
	\centering
	\begin{subfigure}{\linewidth}
		\includegraphics[width=\linewidth]{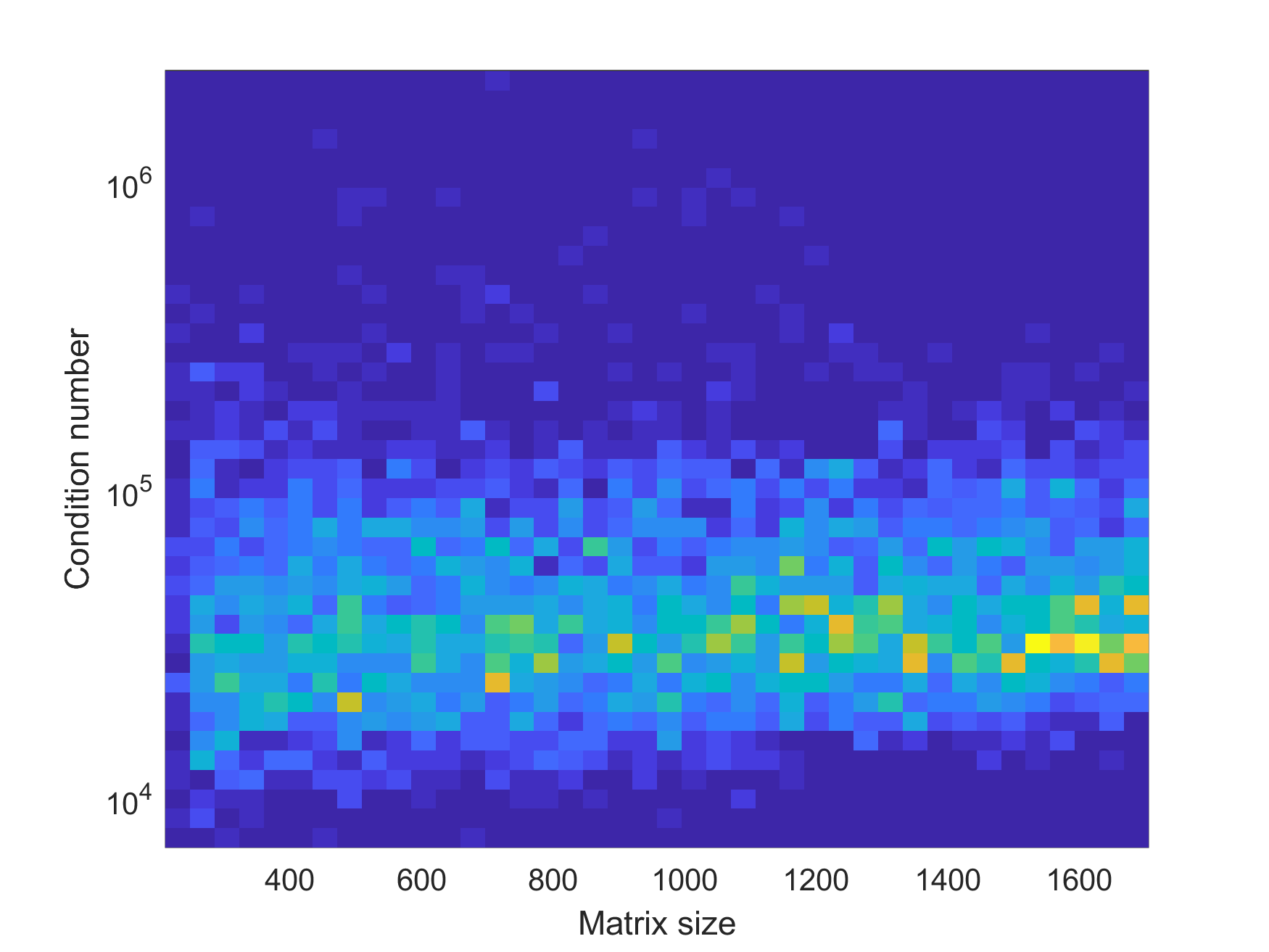}
		\caption{Condition number $\kappa$ for instances of \eqref{prob:experimental portfolio} of different sizes.}
		\label{fig:kappa size evolution}
	\end{subfigure}
	\begin{subfigure}{\linewidth}
		\includegraphics[width=\linewidth]{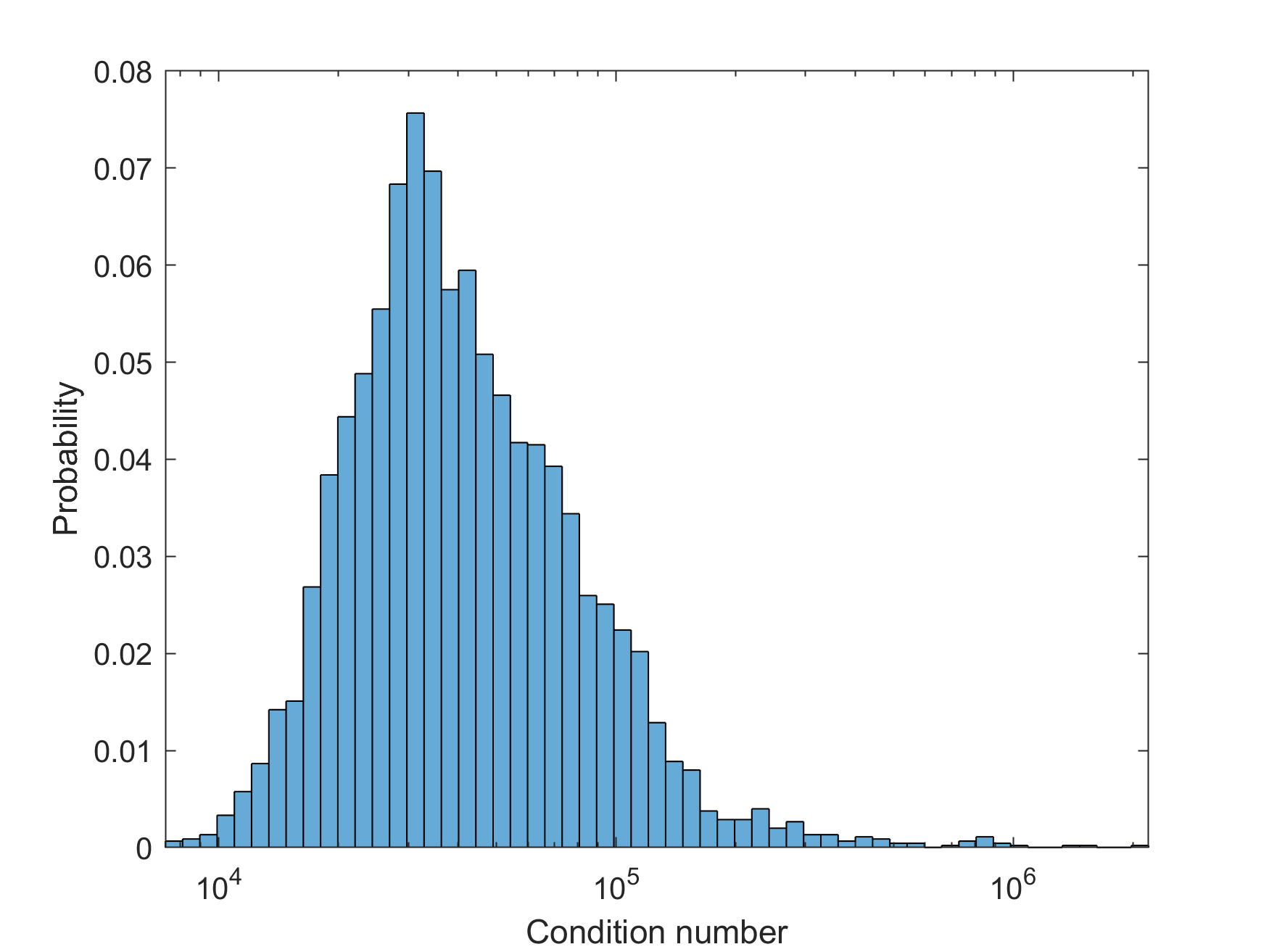}
		\caption{Distribution of $\kappa$ for all instances, of any size.}
		\label{fig:kappa histogram}
	\end{subfigure}
	\caption{Condition number $\kappa$ for different samples of problem \eqref{prob:experimental portfolio}.}
\end{figure}

Figures \ref{fig:kappa size evolution} and \ref{fig:kappa histogram} show the values of the condition number $\kappa$ for different samples of problem \eqref{prob:experimental portfolio}. From Figure \ref{fig:kappa size evolution} we deduce that $\kappa$ does not seem to depend on the problem size, whereas from Figure \ref{fig:kappa histogram} we see that $\kappa$ is very large only in a small number of instances (not necessarily the largest ones). On the other hand, on Figure \ref{fig:delta heatmap} we see that the $1/\delta^2$ term indeed increases with the problem size parameter $n$. By fitting a power law of the form $y = a\cdot x^b$ through these datapoints, we obtain that the exponent $b$ is $1.050$ with a $95\%$ confidence interval of $[0.965, 1.136]$ -- so, $1/\delta^2$ grows roughly linearly with $n$.

\begin{figure}[h]
	\centering 
	\includegraphics[width=\linewidth]{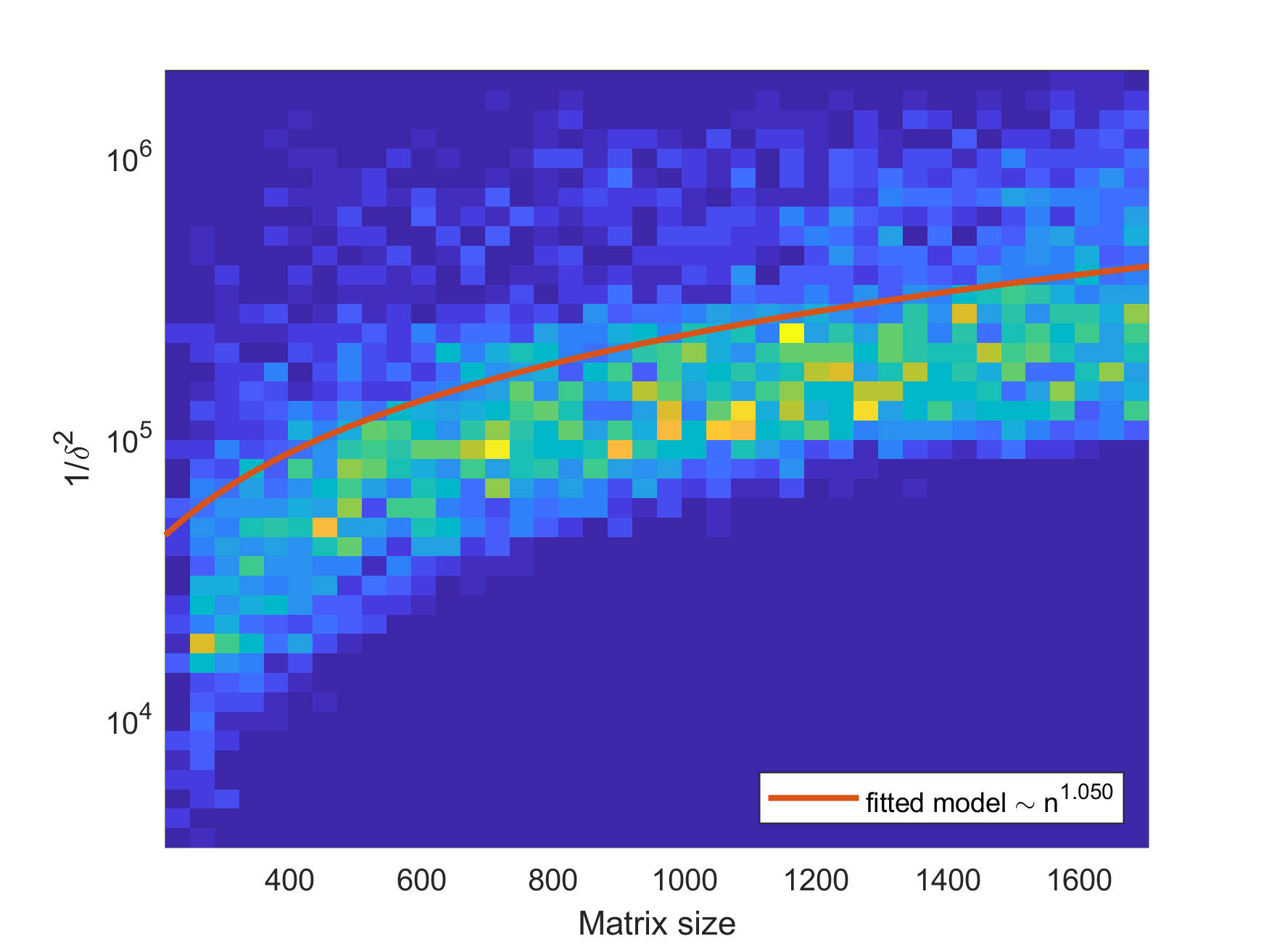}
	\caption{The inverse-square of the tomography precision $\delta$ grows roughly linearly.}
	\label{fig:delta heatmap}
\end{figure}

Finally, we estimate the average-case complexity of Algorithm \ref{alg:qipm} by substituting these observed values of $r, n, \epsilon, \delta, \kappa$ and $\zeta$ into the expression from Theorem \ref{thm:runtime}. Figure \ref{fig:complexity fit} shows how this quantity increases with the problem size $n$, after removing $1\%$ of the biggest outliers. Again, by finding the least-squares fit of a power law through these points, one obtains a dependence of $O(n^{2.387})$, with a $95\%$ confidence bound of $[2.184, 2.589]$. When compared to the classical complexity that scales as $O(n^{\omega + 0.5})$ (which can be thought of as $n^{3.5}$ in practice), we see that for most instances the quantum algorithm achieves a speedup by a factor of almost $O(n)$.

\begin{figure}[h]
	\centering
	\includegraphics[width=\linewidth]{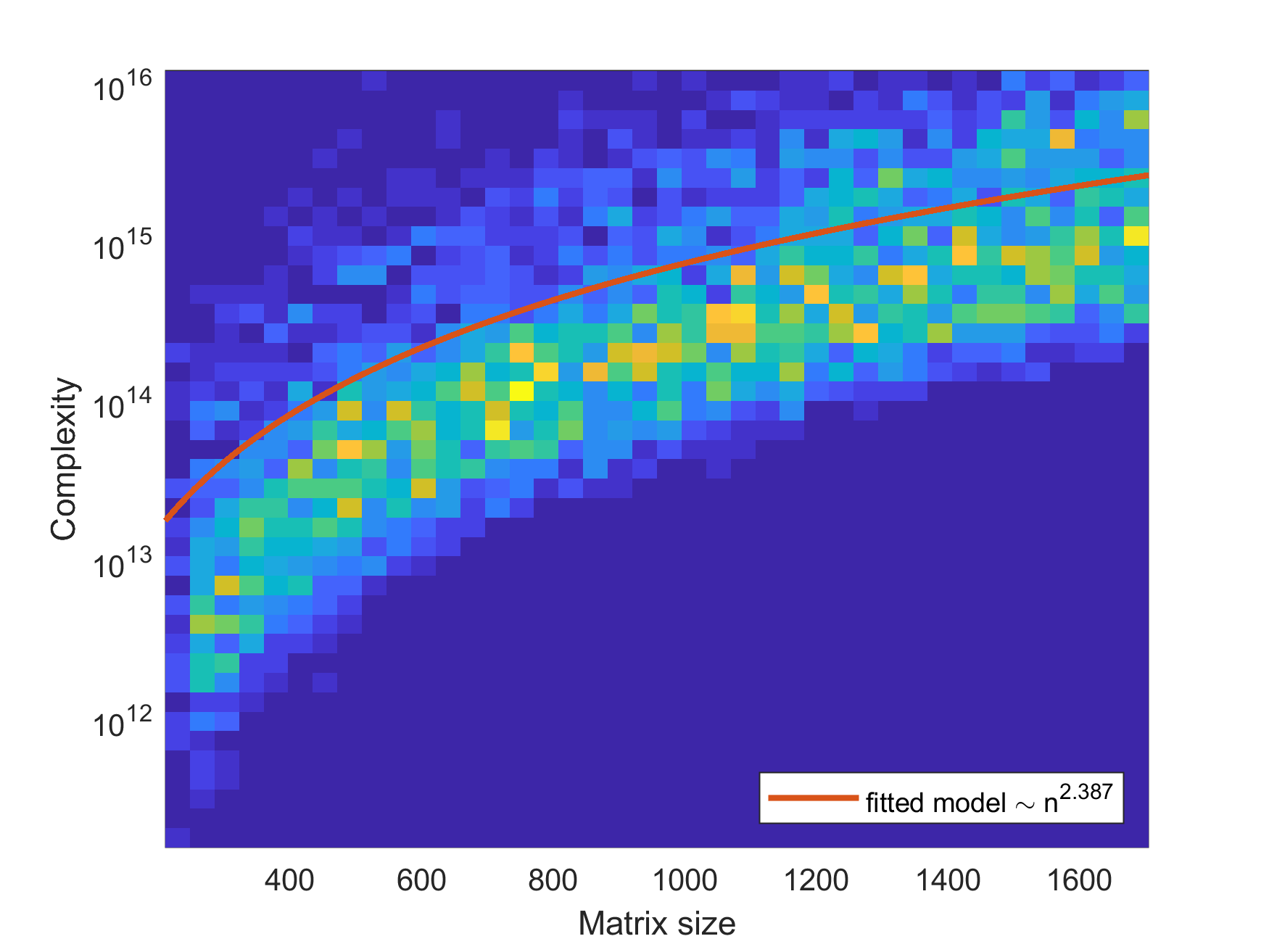}
	\caption{Observed complexities for $\epsilon=0.1$ and the corresponding power-law fit. The $x$-axis is $n$, the size of the Newton system, and the $y$-axis is the observed complexity, as per Theorem \ref{thm:runtime}.}
	\label{fig:complexity fit}
\end{figure}

\section{Concluding remarks}
In this paper we present the first quantum algorithm for the constrained portfolio optimization problem. The algorithm is based on the quantum interior-point method (IPM) framework introduced by \cite{KP18}, and more precisely on the quantum IPM for second-order conic programming (SOCP) presented in \cite{R19}. The problem that is being solved is described by \eqref{prob:portfolio}, and it is more general than the problems considered previously (e.g. by \cite{LR18}), since it allows imposing nonnegativity and budget constraints on the design variables.

The main technical contribution of this paper is Algorithm \ref{alg:qipm}, which solves problem \eqref{prob:portfolio} by reducing it to a SOCP, which can further be solved by using the quantum SOCP IPM from \cite{R19}. Theoretical analysis shows that the running time of this algorithm scales as $\widetilde{O}\left( n^{1.5} \frac{\zeta \kappa}{\delta^2} \log(1/\epsilon) \right)$, which is better than the classical complexity of $\widetilde{O}\left( n^{\omega+0.5}\log(1/\epsilon) \right)$ if the problem-dependent parameters $\kappa$ and $\zeta$ are not ``too large'' and if $\delta$ is not ``too small''. 

To bound these parameters, we simulate the execution of Algorithm \ref{alg:qipm} when solving different instances of problem \eqref{prob:experimental portfolio} to a fixed precision $\epsilon$. These experiments suggest that there is no observable growth in the condition number $\kappa$ as the problem size increases, whereas the most significant impact on the running time of Algorithm \ref{alg:qipm} comes from the factor $1/\delta^2$, which has been observed to grow roughly linearly with the problem size. Finally, using the observed values of these parameters, we estimate the average-case complexity of our algorithm to scale as $O(n^b)$, where the exponent $b$ has a $95\%$ confidence interval of $2.387 \pm 0.202$. Therefore, for most instances, we obtain an almost $O(n)$-speedup over the practical $O(n^{3.5})$ complexity of the classical algorithm.

\textbf{Acknowledgments:} Part of this work was supported by the grants QuantERA QuantAlgo and ANR QuData.